\documentclass{article}
\usepackage{geometry}
\usepackage{titling}
\AtEndDocument{\bigskip{\footnotesize%
\textsc{$^*$School of Computing, Dehradun Institute of Technology, Dehradun, Uttarakhand–248009, India} \par  
  \textit{E-mail address:} \texttt{suryagiri456@gmail.com} \par

  \addvspace{\medskipamount}
}}
\usepackage{doc}

\usepackage{url}

\usepackage{graphicx}
\usepackage{epstopdf}
\usepackage{amsthm}
\usepackage{amssymb}
\usepackage{amsmath}
\usepackage{enumitem}

\usepackage{array}
\usepackage{enumitem}
\usepackage[utf8]{inputenc}
\usepackage[english]{babel}
\newtheorem{theorem}{Theorem}

\newtheorem{lemma}[theorem]{Lemma}

\newtheorem*{remark}{Remark}

\DeclareMathOperator{\RE}{Re}

\usepackage{hypdoc}
\usepackage{varwidth}
\usepackage{lineno}
\usepackage{graphicx}
\usepackage{epstopdf}
\usepackage{amsmath}
\usepackage{amsthm}
\usepackage{enumitem}
\usepackage{amsthm}
\usepackage{float}
\usepackage{subcaption}
\usepackage{caption}
\usepackage{authblk}
\usepackage{abstract}
\usepackage{thmtools}
\declaretheorem[numbered=no,
name=Theorem A]{theoremA}

\begin{document}
\title{Third-Order Toeplitz Determinant for a Subclass of Starlike Mappings in Higher Dimensions}
\author{Surya Giri$^*$  }
%\address{Department of Applied Mathematics, Delhi Technological University, Delhi, SC 29208}
%\email{}
%\author{Surya Giri*}
%\address{}
%\email{surya}
%\author[$\dagger$]{Surya Giri}
%\author[$\star$]{S. Sivaprasad Kumar}
%\author[$\ddag$]{J. Jones}

%%\affil[$\dagger$,$\star$]{Department of Applied Mathematics, Delhi Technological University, Delhi–110042, India}\\
%\vspace{0.5cm}
%\vspace{0.5cm}

  %\textit{E-mail address}, R.~Campbell: \texttt{campr@galois.psu.edu}
%\affil[$\star$]{Atmospheric Research Station,
%Pala Lundi, Fiji}
%\affil[$\ddag$]{Department of Philosophy, Freedman College,
%Periwinkle, Colorado 84320}
\date{}

%\title{Bohr-Rogosinski phenomenon for $\mathcal{S}^*(\psi)$ and $\mathcal{C}(\psi)$}
%	\thanks{K. Gangania thanks to University Grant Commission, New-Delhi, India for providing Junior Research Fellowship under }	

%	\author[Kamaljeet]{Kamaljeet Gangania}
%	\address{Department of Applied Mathematics, Delhi Technological University,
%%	\email{gangania.m1991@gmail.com}
	
%	\author[S. Sivaprasad Kumar]{S. Sivaprasad Kumar}
%	\address{Department of Applied Mathematics, Delhi Technological University,
%		Delhi--110042, India}
%	\email{spkumar@dce.ac.in}

\maketitle	

\begin{abstract}
    \noindent      The manuscript establishes sharp bound of the third-order Toeplitz determinant for a subclass of starlike mappings defined on the unit ball in a complex Banach space and on bounded starlike circular domains in $\mathbb{C}^n.$
\end{abstract}
\vspace{0.5cm}

\section{Introduction}\label{sec1}
    Let $\mathcal{A}$ be the class of analytic functions $f$ defined on the open unit disk $\mathbb{U}= \{\zeta\in \mathbb{C}: \vert \zeta \vert<1\}$ and normalized by the conditions $f(0)=0$ and $f'(0)=1$. Further, let $\mathcal{S}\subset \mathcal{A}$ denote the class of univalent functions in $\mathbb{U}$. The well-known and most studied subclass of $\mathcal{S}$ is the class of starlike functions, denoted by $\mathcal{S}^*$. A function $f\in \mathcal{S}^*$ if and only if $\RE (\zeta f'(\zeta)/f(\zeta))> 0$ for all $\zeta \in \mathbb{U}$.
    Let $\mathcal{P}$ denote the class of analytic functions $p$ in $\mathbb{U}$ such that $p(0)=1$ and $\RE p(\zeta)>0 $ for all $\zeta\in \mathbb{U}$, known as the Carath\'{e}odory class.% Further, let $\mathcal{B}_0$ represent the class of Schwarz functions $\omega$ satisfying $\omega (0)=0$ and $\vert \omega(\zeta)\vert< 1$ for $\zeta\in \mathbb{U}$.

    Toeplitz matrices and their determinants appear in numerous branches of mathematics. A detailed exposition of their applications in both pure and applied mathematics can be found in  the survey by Ye and Lim~\cite{YeLim}. For $f(\zeta)=\zeta+\sum_{n=2}^\infty a_n \zeta^n\in \mathcal{A}$, Ali et al.~\cite{AliThoVas} considered the Toeplitz determinant
\begin{equation*}
     T_{m,n}(f)= \begin{vmatrix}
	a_n & a_{n+1} & \cdots & a_{n+m-1} \\
	a_{n+1} & a_n & \cdots & a_{n+m-2}\\
	\vdots & \vdots & \ddots & \vdots\\
    a_{n+m-1} & a_{n+m-2} & \cdots & a_n\\
	\end{vmatrix}.
\end{equation*}
   In particular, $ T_{2,1}(f) = a_2^2 - a_3^2$,  $ T_{3,1}(f) = 1 - 2 a_2^2 + 2 a_2^2 a_3 -  a_3^2  $  and
\begin{equation}
      T_{3,2}(f)  =  (a_2 - a_4) (a_2^2 - 2 a_3^2 +a_2 a_4) .
\end{equation}
     Ali et al.~\cite{AliThoVas} derived the sharp estimates of Toeplitz determinants for the classes $\mathcal{S}$, $\mathcal{S}^*$ and some other subclasses of $\mathcal{S}$ for the initial values of $m$ and $n$. Motivated by their work, several researchers established sharp bounds of certain Toeplitz and Hermitian-Toeplitz determinants for various subclasses of $\mathcal{A}$~\cite{AhuKhaRav,Cudna,GirKum1}. Among other results,  Ali et al.~\cite{AliThoVas} derived the following:
\begin{theoremA}\label{thmABC}\cite{AliThoVas}
   If $f\in \mathcal{S}^*$, then $\vert T_{3,2}(f) \vert \leq 84$. The estimate is sharp.
\end{theoremA}

   Here, we generalize this result to higher dimensions for a subclass of starlike mappings defined on the unit ball in a complex Banach space and on bounded starlike circular domains in $\mathbb{C}^n$. Let $X$ denote a complex Banach space equipped with a norm $\| \cdot\|$ and $\mathbb{B}$ be the unit ball in $X$.  We denote by $L(X,Y)$ the collection of all continuous linear operators from $X$ to another Banach space $Y.$ For each $z\in X\setminus\{0\}$, define the set
     $$ T_z = \{ l_z \in L(X,\mathbb{C}) : l_z(z) = \| z\|, \| l_z \| = 1\}.$$
   This set is non-empty by  virtue of the Hahn-Banach theorem. Let $\mathcal{H}(\Omega,\Omega')$ represent the set of all holomorphic mappings from a domain $\Omega\subset X$ into another domain $\Omega'\subset Y$ and $\partial \Omega$ denote the boundary of $\Omega$. We simply take $\mathcal{H}(\Omega,X)=\mathcal{H}(\Omega)$. It is well known that if $f\in \mathcal{H}(\mathbb{B})$, then
    $$f(w)= \sum_{n=0}^\infty \frac{1}{n!} D^n f(z)((w-z)^n),$$
    for all $w$ in some neighbourhood of $z\in \mathbb{B}$, where $D^n f(z)$ is a bounded symmetric $n-$linear mapping from $\prod_{j=1}^n X$ into $X$, called the $n^{th}$ Fr\'{e}chet derivative of $f$ at $z$. Moreover,
        $$ D^n f(z)((w-z)^n)= D^n f(z) \underbrace{( w-z, w-z, \cdots, w-z) }_\text{ n -times}.$$
   A mapping $f\in \mathcal{H}(\mathbb{B})$ is said to be biholomorphic if its inverse exists and is holomorphic on the domain $f(\mathbb{B})$. Furthermore, if the derivative $D f(z)$ of $f\in \mathcal{H}(\mathbb{B})$ has a bounded inverse for each $z\in \mathbb{B}$, then $f$ is said to be locally biholomorphic. In analogy with the one-dimensional case, a mapping $f\in \mathcal{H}(\mathbb{B})$ is said to be normalized if $f(0)=0$ and $Df(0)=I$, where $I$ is the linear identity operator from $X$ into $X$. The class of all normalized biholomorphic mappings on $\mathbb{B}$ is represented by $\mathcal{S}(\mathbb{B})$. A biholomorphic mapping on $\mathbb{B}$ is said to be starlike if $f(\mathbb{B})$ is starlike with respect to the origin.

   Let $\Omega\subset \mathbb{C}^n$ represent a bounded starlike circular domain containing the origin such that its Minkowski functional $\rho(z)\in \mathcal{C}^1$ except for some lower dimensional manifolds in $\mathbb{C}^n$. The first and the $m^{th}$ Fr\'{e}chet derivative of a holomorphic mapping $f \in \mathcal{H}(\Omega)$  are written by
    $ D f(z)$ and $D^m f(z) (a^{m-1},\cdot)$, respectively. The matrix representations are
\begin{align*}
    D f(z) &= \bigg(\frac{\partial f_j}{\partial z_k} \bigg)_{1 \leq j, k \leq n}, \\
    D^m f(z)(a^{m-1}, \cdot) &= \bigg( \sum_{p_1,p_2, \cdots, p_{m-1}=1}^n  \frac{ \partial^m f_j (z)}{\partial z_k \partial z_{p_1} \cdots \partial z_{p_{m-1}}} a_{p_1} \cdots a_{p_{m-1}}   \bigg)_{1 \leq j,k \leq n},
\end{align*}
   where $f(z) = (f_1(z), f_2(z), \cdots f_n(z))'$ and $ a= (a_1, a_2, \cdots a_n)'\in \mathbb{C}^n.$ Let $\mathcal{S}^*(\mathbb{B})$ (respectively, $\mathcal{S}^*(\Omega)$) denote the class of normalized starlike mappings defined on $\mathbb{B}$ (respectively, on $\Omega$).

   Cartan~\cite{Cart} showed that the most famous Bieberbach conjecture for the class $\mathcal{S}$ does not hold in higher dimensions. Various counterexamples demonstrate that several results of geometric function theory in one complex variable cannot be directly extended to higher dimensions, at least without imposing additional restrictions~\cite{GraKoh}. For some work related to coefficient inequalities in higher dimensions, we refer~\cite{GraHam,HamHon,Kohr,LuoXu}.  Recently, the bounds of certain Toeplitz determinants for various subclasses of biholomorphic mappings are derived in higher dimensions~\cite{Giri1,Giri3,GirKum2}. In particular, Giri and Kumar~\cite{GirKum2,GirKum3} obtained the sharp estimates of $\vert T_{2,2}(f)\vert$ and $\vert T_{3,1}(f)\vert$ for starlike mappings on the unit ball in $X$ and on bounded starlike circular domains $\Omega$. However, the problem of estimating $\vert T_{3,2}(f)\vert$ remains open. Addressing this gap, we derive the sharp bound of $\vert T_{3,2}(f)\vert$ for certain subclasses of $\mathcal{S}^*(\mathbb{B})$ and $\mathcal{S}^*(\Omega)$, thereby extending Theorem \hyperref[thmABC]{A} to higher dimensions.

   The following lemmas are used in the proofs of the desired results.
\begin{lemma}\cite{Suf}
    Let $f\in\mathcal{H}(\mathbb{B})$ be a normalized locally biholomorphic mapping. Then $f$ is  starlike on $\mathbb{B}$ if and only if
   $$ \RE (l_z ([Df(z)]^{-1}f(z)))> 0, \quad z\in \mathbb{B}\setminus \{0\}, \quad l_z \in T_z. $$
\end{lemma}
\begin{lemma}\label{Minkow}\cite{LiuRen}
 A domain $\Omega \subset \mathbb{C}^n$  is called a bounded starlike circular domain if and only if there exists a unique continuous function
 $\rho: \mathbb{C}^n \rightarrow \mathbb{R},$ known as the Minkowski functional of $\Omega$, such that
\begin{enumerate}
    \item[(i)] $\rho(z) \geq 0$, $z \in \mathbb{C}^n$; \quad $\rho(z) = 0  \Leftrightarrow z=0$,
    \item[(ii)] $\rho(t z) = \vert t\vert , \rho(z)$, $t \in \mathbb{C}$, $z \in \mathbb{C}^n$,
    \item[(iii)] $\Omega = \{ z \in \mathbb{C}^n : \rho(z) < 1 \}$.
\end{enumerate}
   Furthermore, if $\rho(z)$, $z \in \Omega$, belongs to $\mathcal{C}^1$ except on some lower-dimensional manifold $E \subset \mathbb{C}^n$, then $\rho(z)$ satisfies the following properties:
\begin{align}\label{LmEqn}
   & 2 \frac{\partial \rho(z)}{\partial z} z = \rho(z), \quad z \in \mathbb{C}^n \setminus E, \\
    & 2 \frac{\partial \rho(z)}{\partial z} \bigg\vert_{z=z_0} = 1, \quad z_0 \in \partial \Omega \setminus E, \notag\\
    & \frac{\partial \rho(\lambda z)}{\partial z} = \frac{\partial \rho(z)}{\partial z}, \quad \lambda \in (0,\infty),\; z \in \mathbb{C}^n \setminus E, \notag\\
    & \frac{\partial \rho(e^{i \theta} z)}{\partial z} = e^{-i\theta} \frac{\partial \rho(z)}{\partial z}, \quad \theta \in \mathbb{R},\ z \in \mathbb{C}^n \setminus E, \notag
\end{align}
where
   $\frac{\partial \rho(z)}{\partial z} = \left( \frac{\partial \rho(z)}{\partial z_1}, \dots, \frac{\partial \rho(z)}{\partial z_n} \right).$
\end{lemma}

\begin{lemma}\cite{LiuRen}
   Let $\Omega \subset \mathbb{C}^n$ be a bounded starlike circular domain with $0\in \Omega$, whose Minkowski functional $\rho(z)$ belongs to $\mathcal{C}^1$ except on some lower-dimensional manifolds $E \subset \mathbb{C}^n$. If $f: \Omega \rightarrow \mathbb{C}^n$ is a normalized locally biholomorphic mapping, then $f$ is starlike on $\Omega$ if and only if
  $$\RE \Big( \frac{\partial \rho(z)}{\partial z}\,(Df(z))^{-1} f(z) \Big) > 0,\quad z \in \Omega \setminus E. $$
  %The class of all starlike mappings on $\Omega$ is denoted by $\mathcal{S}^*(\Omega)$.
\end{lemma}
\begin{lemma}\cite{Efra}\label{lm1}
   If $p(z)=1+ \sum_{n=1}^\infty p_n z^n \in \mathcal{P}$, then
   $$ \vert p_n \vert\leq 2 \;\; \text{and}\;\;\vert p_n - \mu p_k p_{n-k} \vert \leq 2 \max \{ 1, \vert 2\mu -1 \vert \} , \quad 1\leq k \leq n-1,\;\; n\in \mathbb{N}.$$
  The inequalities are sharp. In particular, if $\vert 2\mu -1 \vert \geq 1$, equality holds for $p(z)=(1+z)/(1-z)$, whereas if $\vert 2\mu-1\vert <1$, inequality is attained for $p(z)=(1+z^n)/(1-z^n)$.
\end{lemma}
\section{Main results}
    In this section, we extend Theorem \hyperref[thmABC]{A} to the case of several complex variables.%, we derive sharp bounds of $\vert T_{3,2}(f)\vert$ for  subclasses $\mathcal{S}^*(\mathbb{B})$ and $\mathcal{S}^*(\Omega)$.
\begin{theorem}\label{thm1}
     Let $f\in \mathcal{H}(\mathbb{B},\mathbb{C})$ with $f(0)=1$ and suppose that $F(z)=z f(z)$. If $F\in \mathcal{S}^*(\mathbb{B})$, then
    $$\vert (A_2 -A_4 )(A_2^2-2 A_3^2 + A_2 A_4) \vert \leq 84,$$
    where
\begin{equation}\label{Ais}
    A_2=  \frac{ l_z (D^2 F(0) (z^2))}{2! \| z \|^2},\;\; A_3= \frac{ l_z (D^3 F(0) (z^3))}{3! \| z \|^3} ,\;\;  A_4 =\frac{ l_z (D^4 F(0) (z^4))}{4! \| z \|^4}, \;\; z\in \mathbb{B}\setminus\{0\}
\end{equation}
   and $l_z\in T_z$.  Moreover, the estimate is sharp.
\end{theorem}
\begin{proof}
    Let $z_0 = \frac{z}{\|z \|}$ for  fixed $z \in  X\setminus \{ 0 \} $. Consider the function $ g : \mathbb{U} \rightarrow \mathbb{C}$ defined by
\begin{equation*}
    g(\zeta) = \left\{ \begin{array}{ll}
     \dfrac{\zeta}{ l_z ((D F(\zeta z_0))^{-1} F( \zeta z_0) )}, & \zeta \neq 0, \\ \\
    1, & \zeta =0.
    \end{array}
    \right.
\end{equation*}
   Hence, $g \in \mathcal{H}(\mathbb{U})$ and for $\zeta \in \mathbb{U}$, it follows that
\begin{align*}
     g(\zeta) &= \frac{\zeta}{ l_z ((D F(\zeta z_0))^{-1} F( \zeta z_0) )} \\
          &= \frac{\zeta}{ l_{z_0} ((D F(\zeta z_0))^{-1} F( \zeta z_0) )}  = \frac{\| \zeta z_0 \|}{ l_{\zeta z_0} ((D F(\zeta z_0))^{-1} F( \zeta z_0) )}.
\end{align*}
    Since $ F \in \mathcal{S}^*(\mathbb{B})$, we have
\begin{equation}
   \RE g(\zeta) > 0, \quad \zeta \in \mathbb{U}.
\end{equation}
    Using the same approach as in \cite[Theorem 7.1.14]{GraKoh}, we deduce that
    $$ (D F(z))^{-1} = \frac{1}{f(z)} \bigg( I - \frac{\frac{z D f(z)}{f(z)}}{1 + \frac{D f(z) z}{f(z)}} \bigg),  \quad z\in \mathbb{B}. $$
    Thus, we obtain
    $$ (D F(z))^{-1} F(z) =  \frac{ z f(z) }{f(z) + D f(z) z}, $$
    which yields
\begin{equation}\label{newe}
   \frac{\| z\|}{l_z ((D F(z))^{-1} F(z))} = 1 + \frac{D f(z) z}{f(z)} .
\end{equation}
   In view of (\ref{newe}), we deduce that
\begin{equation*}\label{accr}
    g(\zeta) = \frac{\| \zeta z_0 \| }{l_{ \zeta z_0} ((DF(\zeta z_0))^{-1} F( \zeta z_0) ) }  =  1 + \frac{D f(\zeta z_0)\zeta z_0}{f(\zeta z_0)}.
\end{equation*}
   Using the Taylor expansions of $g(\zeta)$ and $f(\zeta z_0)$, we get
\begin{align*}
   \bigg(1 + & g'(0) \zeta  + \frac{g''(0)}{2} \zeta^2 + \cdots \bigg)\bigg( 1 + Df(0)(z_0) \zeta + \frac{ D^2 f(0)(z_{0}^2)}{2} \zeta^2 + \cdots \bigg)  \\
   & =\bigg( 1 + Df(0)(z_0) \zeta + \frac{ D^2 f(0)(z_{0}^2)}{2} \zeta^2 + \cdots \bigg) + \bigg( Df(0)(z_0) \zeta +  D^2 f(0)(z_{0}^2)\zeta^2 + \cdots \bigg).
\end{align*}
    Comparison of the homogeneous expansions leads us to
\begin{align*}
    g'(0) &= D f(0)(z_0),\;\; \frac{g''(0)}{2} =  D^2 f(0)(z_0^2) - (D f(0)(z_0))^2
\end{align*}
   and
   $$     \frac{g'''(0)}{6} = ( Df(0)(z_0))^3 - \frac{3}{2} D f(0)(z_0) D^2 f(0) (z_0^2) + \frac{D^3 f(0) (z_0^3)}{2},$$
   which further yield
\begin{equation}\label{eqhf2}
     g'(0) \|z\|= D f(0)(z), \;\;    \frac{g''(0)}{2}\| z\|^2 =  D^2 f(0)(z^2) - (D f(0)(z))^2
\end{equation}
   and
\begin{equation}\label{eqhf3}
    \frac{g'''(0)}{6} \|z\|^3 =  ( Df(0)(z))^3 -\frac{3}{2} D f(0)(z) D^2 f(0) (z^2) + \frac{D^3 f(0) (z^3)}{2},
\end{equation}
    respectively. Furthermore, from the identity $F(z) = z f(z)$, we have
\begin{align*}
     \frac{ D^2 F(0) (z^2)}{2! } &=  D f(0)(z)  z,\;\;  \frac{ D^3 F(0) (z^3)}{3! } =  \frac{ D^2 f(0) (z^2)}{2! } z\;\; \text{and}\;\;  \frac{ D^4 F(0) (z^4)}{4! } =  \frac{ D^3 f(0) (z^3)}{3! } z,
\end{align*}
   which immediately imply
\begin{align*}
     \frac{l_z( D^2 F(0) (z^2))}{2! } &=  D f(0)(z)  \|z\|, \;\;\frac{l_z (D^3 F(0) (z^3))}{3! } =  \frac{ D^2 f(0) (z^2)}{2! } \|z\|
\end{align*}
 and
   $$          \frac{l_z( D^4 F(0) (z^4))}{4! } =  \frac{ D^3 f(0) (z^3)}{3! } \|z\|,$$
   respectively.   By virtue of (\ref{eqhf2}) and (\ref{eqhf3}),  the above expressions become
\begin{equation}\label{T1E3B}
     \frac{l_z( D^2 F(0) (z^2))}{2! \|z\|^2} = g'(0), \;\;    \frac{l_z (D^3 F(0) (z^3))}{3! \|z\|^3} = \frac{1}{2}\bigg(\frac{g''(0)}{2}+ (g'(0))^2 \bigg)
\end{equation}
   and
\begin{equation}\label{T1E4B}
    \frac{l_z( D^4 F(0) (z^4))}{4! \|z\|^4} =  \frac{1}{3} \bigg( \frac{g'''(0)}{6}+\frac{(g'(0))^3}{2}+ \frac{3 g'(0) g''(0)}{4} \bigg),
\end{equation}
   respectively.
    Using (\ref{T1E3B}), (\ref{T1E4B})  together with (\ref{Ais}), we obtain
\begin{align*}
     \vert  A_2^2-2 A_3^2+A_2 A_4 \vert  &= \bigg\vert (g'(0))^2 - \frac{(g'(0))^4}{3} - \frac{(g'(0))^2 g''(0)}{4} - \frac{(g''(0))^2}{8}+  \frac{g'(0) g'''(0)}{18}  \bigg\vert \\
     &\leq  \vert g'(0) \vert^2 + \frac{\vert g'(0) \vert^4}{3}  + \frac{1}{2}\bigg\vert \frac{g''(0)}{2} \bigg\vert^2  +\frac{1}{3} \vert g'(0) \vert\bigg \vert \frac{g'''(0)}{6}  - \frac{3}{2} \bigg(\frac{g'(0) g''(0)}{2} \bigg) \bigg\vert.
\end{align*}
   Since $g\in \mathcal{P}$, applying Lemma~\ref{lm1} to the above inequality, it follows that
\begin{equation}\label{A2A3A4}
    \vert  A_2^2-2 A_3^2+A_2 A_4 \vert \leq 14.
\end{equation}
   Again, using the estimate $\vert g'(0)\vert \leq 2$ from Lemma~\ref{lm1} in (\ref{T1E3B}), we deduce that
\begin{equation}\label{A2}
   \vert A_2 \vert =\Big\vert   \frac{l_z( D^2 F(0) (z^2))}{2! \|z\|^2} \Big\vert  \leq 2.
\end{equation}
   By the one-to-one correspondence between $\mathcal{P}$ and the class of Schwarz functions, there exists a function $\omega$ of the form $\omega(\zeta)=\sum_{n=1}^\infty c_n \zeta^n $ satisfying $\vert\omega(\zeta)\vert<1 $ and $\omega(0)=0$ such that
   $$ g(\zeta)=\frac{1+ \omega(\zeta)}{1- \omega(\zeta)}, \quad \zeta \in \mathbb{U}. $$
   Comparing the coefficients of same powers of $\zeta$ on the both side leads us to
   $$  g'(0)= 2 c_1 ,\quad \frac{g''(0)}{2}= 2(c_1^2+ c_2) \quad \text{and}\quad  \frac{g'''(0)}{6} = 2 (c_1^3 + 2 c_1 c_2 + c_3). $$
   Consequently, from (\ref{T1E4B}), we have
   $$ \bigg\vert \frac{l_z( D^4 F(0) (z^4))}{4! \|z\|^4}\bigg\vert = \frac{2}{3} \left\vert 6 c_1^3 + 5 c_1 c_2+ c_3 \right\vert.  $$
   In view of the bound proved in~\cite[Lemma 2]{ProSzy}, we get
\begin{equation}\label{A4}
     \vert A_4 \vert = \bigg\vert \frac{l_z( D^4 F(0) (z^4))}{4! \|z\|^4}\bigg\vert \leq 4.
\end{equation}
    The bounds given by (\ref{A2}) and (\ref{A4}) together with the triangle inequality yield
\begin{equation}\label{A2A4}
     \vert A_2 - A_4  \vert \leq   \vert A_2 \vert + \vert A_4\vert \leq 6.
\end{equation}
   Combining the estimates from~(\ref{A2A3A4}) and (\ref{A2A4}), we deduce that
   $$\vert  (A_2 - A_4)( A_2^2-2 A_3^2+A_2 A_4)  \vert \leq ( \vert A_2\vert + \vert A_4 \vert) \vert A_2^2-2 A_3^2+A_2 A_4  \vert \leq 84, $$
   which is the desired bound.

   To prove the inequality is sharp, we consider the mapping $F$ defined as
\begin{equation}\label{extB}
    F(z) = \frac{z}{(1-i \, l_{u}(z))^2}, \quad z\in \mathbb{B}, \quad \| u \|=1.
\end{equation}
  We deduce that $ F \in \mathcal{S}^*(\mathbb{B})$ and a direct computation gives
  $$  \frac{D^2  F(0) (z^2)}{2!}=  2 i l_u(z) z , \;\; \frac{D^3  F(0) (z^3)}{3!} =- 3 (l_u (z))^2 z \;\; \text{and}\;\;   \frac{D^4  F(0) (z^4)}{4!}  =- 4 i (l_u(z))^3 z, $$
   which yields
  $$  \frac{l_z(D^2  F(0) (z^2))}{2!}=  2 i l_u(z) \|z\|, \;\;\; \frac{l_z (D^3  F(0) ( z^3 ))}{3!} = - 3 (l_u (z))^2 \| z \|  $$
   and
\begin{align*}
     \frac{l_z (D^4  F(0) (z^4))}{4!} & =- 4 i (l_u(z))^3 \|z\|,
\end{align*}
   respectively. Setting $z = r u$ $(0< r <1)$, we obtain
\begin{equation*}\label{cftB}
     A_2= \frac{l_z(D^2  F(0) (z^2))}{2! \|z\|^2}=  2 i, \;\;\; A_3 =\frac{l_z (D^3  F(0) ( z^3) ) }{3! \| z \|^3}  =  - 3 \;\; \text{and}\;\;  A_4= \frac{l_z (D^4  F(0) (z^4))}{4! \| z \|^4}  =- 4i.
\end{equation*}
   Consequently, for the mapping $F$, we have
\begin{align*}
    \Big\vert (A_2 -A_4) ( A_2^2-2 A_3^2+A_2 A_4)\Big\vert  = 84,
\end{align*}
  which confirms the sharpness of the bound.
\end{proof}
\begin{remark}
    When $X=\mathbb{C}$ and $\mathbb{B}=\mathbb{U}$, Theorem~\ref{thm1} is equivalent to Theorem \hyperref[thmABC]{A}.
\end{remark}
\begin{theorem}\label{thm2}
     Let $f \in \mathcal{H}(\Omega, \mathbb{C})$ with $f(0)=1$  and suppose that $F(z) =  z f(z)$. If $F(z) \in \mathcal{S}^*(\Omega)$,
   then
\begin{equation*}\label{mnresult}
\begin{aligned}
   \vert (A_2 - A_4)(A_2^2-2 A_3^2+A_2 A_4)\vert \leq 84 , \quad z \in \Omega\setminus E ,
\end{aligned}
\end{equation*}
  where
\begin{equation}\label{Biis}
    A_2=  2 \frac{\partial \rho(z)}{\partial z}\frac{D^{2} F(0)(z^{2})}{2! \rho^{2}(z)} ,\;\; A_3= 2 \frac{\partial \rho(z)}{\partial z}\frac{D^{3} F(0)(z^{3})}{3! \rho^{3}(z)} \;\; \text{and}\;\;  A_4 =2 \frac{\partial \rho(z)}{\partial z}\frac{D^{4} F(0)(z^{4})}{4! \rho^{4}(z)}.
\end{equation}
   The bound is sharp.
\end{theorem}
\begin{proof}
  Fix $z \in \Omega \setminus E $ and set $z_0 = \frac{z}{\rho(z)}$. Define a function $h : \mathbb{U} \rightarrow \mathbb{C}$ by
\begin{equation}\label{hkzeta}
   h (\zeta) =
\left\{
\begin{array}{ll}
    \dfrac{\zeta }{2 \frac{\partial \rho(z_0)}{\partial z} (D F(\zeta z_0))^{-1} F(\zeta z_0)}, & \zeta \neq 0,\\
     1 , & \zeta =0.
\end{array}
\right.
\end{equation}
  Clearly, $h\in \mathcal{H}(\mathbb{U})$. Moreover, since $F \in \mathcal{S}^*(\Omega)$, we have
\begin{align*}
    \RE h (\zeta) &= \RE \bigg( \frac{\zeta }{2 \frac{\partial \rho(z_0)}{\partial z} (D F(\zeta z_0))^{-1} F(\zeta z_0)}\bigg)  \\
                  &= \RE \bigg( \frac{\rho (\zeta z_0) }{2 \frac{\partial \rho(\zeta z_0)}{\partial z} (D F(\zeta z_0))^{-1} F(\zeta z_0)}\bigg)> 0,\quad  \zeta \in \mathbb{U}.
\end{align*}
   Proceeding as in Theorem~\ref{thm1}, we obtain
   $$ (D F(z))^{-1} F(z) =\frac{z f(z) }{f(z) + D f(z) z} ,  \quad z \in \Omega \setminus\{0\},$$
   From this relation, it follows that
\begin{equation}\label{thm2eq1}
    \frac{\rho ( z) }{2 \frac{\partial \rho(z)}{\partial z} (D F(z))^{-1} F(z)} = 1 + \frac{D f(z) z}{f(z)} , \quad z \in \Omega\setminus E.
\end{equation}
    In view of (\ref{thm2eq1}), we deduce that
    $$ h(\zeta)=  1 + \frac{D f(\zeta z_0) \zeta z_0}{f(\zeta z_0)}. $$
%   or equivalently,
%   $$  s (\zeta) g(\zeta z_0) =  h(\zeta z_0) + D h(\zeta z_0) \zeta z_0 . $$
   Expanding $f$ and $h$ in Taylor series and comparing corresponding homogeneous expansions, we get
%   Expanding $f$ and $h$ in Taylor series and equating like homogeneous terms, we obtain
\begin{equation*}
    h' (0) = D f(0)(z_0), \quad \frac{h'' (0)}{2} = D^2 f(0) (z_0^2 ) - (D f(0) (z_0))^2
\end{equation*}
  and
\begin{equation*}
    \frac{h'''(0)}{6} = ( Df(0)(z_0))^3 - \frac{3}{2}D f(0)(z_0) D^2 f(0) (z_0^2)  + \frac{D^3 f(0) (z_0^3)}{2},
\end{equation*}
   which further gives
\begin{equation}\label{use0}
    h' (0) \rho(z)= D f(0)(z), \quad \frac{h'' (0)}{2} \rho^2(z) = D^2 f(0) (z^2 ) - (D f(0) (z))^2
\end{equation}
  and
\begin{equation}\label{hing4}
    \frac{h'''(0)}{6} \rho^3(z)= ( Df(0)(z))^3 - \frac{3}{2}D f(0)(z) D^2 f(0) (z^2)  + \frac{D^3 f(0) (z^3)}{2},
\end{equation}
  respectively. In addition, the identity $F(z) =  z f(z)$ yields
\begin{equation*}\label{use1}
      \frac{D^2 F(0) (z^2)}{2!}= Df(0)(z) z, \;\;   \frac{D^3 F(0) (z^3)}{3!} = \frac{D^2 f(0) (z^2)}{2!} z \;\;\text{and}\;\; \frac{D^4 F(0) (z^4)}{4!} =\frac{ D^3 f(0) (z^3)}{3! }  z.
\end{equation*}
   By virtue of~(\ref{LmEqn}), these relations imply
\begin{equation}\label{use3}
     2 \frac{\partial \rho}{\partial z} \frac{D^2 F(0) (z^2)}{2!}= Df(0)(z) \rho(z), \quad     2 \frac{\partial \rho}{\partial z}\frac{D^3 F(0) (z^3)}{3!} = \frac{D^2 f(0) (z^2)}{2!} \rho(z)
\end{equation}
    and
\begin{equation}\label{use4}
      2 \frac{\partial \rho}{\partial z} \frac{D^4 F(0) (z^4)}{4!} =\frac{ D^3 f(0) (z^3)}{3! }  \rho(z),
\end{equation}
   respectively. From (\ref{use0}), (\ref{hing4}), (\ref{use3}) and (\ref{use4}), it follows that
\begin{equation}\label{T1E3}
     2 \frac{\partial \rho(z)}{\partial z}\frac{D^{2} F(0)(z^{2})}{2! \rho^{2}(z)} = h'(0), \;\;    2 \frac{\partial \rho(z)}{\partial z}\frac{D^{3} F(0)(z^{3})}{3! \rho^{3}(z)} = \frac{1}{2}\bigg(\frac{h''(0)}{2}+ (h'(0))^2 \bigg)
\end{equation}
   and
\begin{equation}\label{T1E4}
   2 \frac{\partial \rho(z)}{\partial z}\frac{D^{4} F(0)(z^{4})}{4! \rho^{4}(z)}=  \frac{1}{3} \bigg( \frac{h'''(0)}{6}+\frac{(h'(0))^3}{2}+ \frac{3 h'(0) h''(0)}{4} \bigg).
\end{equation}
   Since $\RE h(\zeta)>0$, applying the bound $\vert h'(0)\vert \leq 2$ from Lemma~\ref{lm1} in (\ref{T1E3}), we obtain
\begin{equation}\label{B2}
    \vert A_2\vert= \bigg\vert 2 \frac{\partial \rho(z)}{\partial z}\frac{D^{2} F(0)(z^{2})}{2! \rho^{2}(z)} \bigg\vert \leq  2.
\end{equation}
    From (\ref{T1E3}) and (\ref{T1E4}), it follows that
\begin{align*}
     \vert  A_2^2-2 A_3^2+A_2 A_4 \vert  &= \bigg\vert (h'(0))^2 - \frac{(h'(0))^4}{3} - \frac{(h'(0))^2 h''(0)}{4} - \frac{(h''(0))^2}{8}+  \frac{h'(0) h'''(0)}{18}  \bigg\vert \\
     &\leq  \vert h'(0) \vert^2 + \frac{\vert h'(0) \vert^4}{3}  + \frac{1}{2}\bigg\vert \frac{h''(0)}{2} \bigg\vert^2  +\frac{1}{3} \vert h'(0) \vert\bigg \vert \frac{h'''(0)}{6}  - \frac{3}{2} \bigg(\frac{h'(0) h''(0)}{2} \bigg) \bigg\vert,
\end{align*}
   where $A_2$, $A_3$ and $A_4$ are given by (\ref{Biis}).  An application of Lemma~\ref{lm1} to the above inequality yields
\begin{equation}\label{A2A3A4B}
     \vert  A_2^2-2 A_3^2+A_2 A_4 \vert \leq 14.
\end{equation}
   Since $h\in \mathcal{P}$, there exists a Schwarz function $\omega(\zeta)=\sum_{n=1}^\infty c_n \zeta^n$ such that $h(\zeta)= (1+\omega(\zeta))/(1- \omega(\zeta))$. In view of this representation,  (\ref{T1E4}) reduces to
\begin{equation}
   2 \frac{\partial \rho(z)}{\partial z}\frac{D^{4} F(0)(z^{4})}{4! \rho^{4}(z)}=   \frac{2}{3} \Big(6 c_1^3 + 5 c_1 c_2+ c_3 \Big).
\end{equation}
    Using the bound obtained in~\cite[Lemma 2]{ProSzy}, it follows that
\begin{equation}\label{B4}
  \vert A_4\vert = \bigg\vert 2 \frac{\partial \rho(z)}{\partial z}\frac{D^{4} F(0)(z^{4})}{4! \rho^{4}(z)}\bigg\vert\leq   4.
\end{equation}
    Applying the estimates given in (\ref{B2}), (\ref{A2A3A4B}) and (\ref{B4}) to the inequality
    $$ \vert  (A_2 - A_4)( A_2^2-2 A_3^2+A_2 A_4)  \vert \leq ( \vert A_2 \vert + \vert A_4 \vert ) \vert A_2^2-2 A_3^2+A_2 A_4 \vert, $$
    the asserted bound follows.

      To show that the estimate is sharp, we consider the mapping
\begin{equation}\label{ExtOmega}
   F(z)= \frac{z}{\Big(1 -  i\left(\dfrac{z_1}{r} \right)\Big)^{2}}, \quad z\in \Omega,
\end{equation}
   where $r=\sup\{\vert z_1\vert : z=(z_1,0,\cdots,0)'\in \Omega \}$. According to~\cite{LiuLiu2}, this mapping belongs to $\mathcal{S}^*(\Omega)$. A direct computation gives
    $$  \frac{D^{2}  F(0) (z^{2})}{2!}= 2 i  \Big(\frac{z_1}{r}\Big) z, \;\;  \frac{D^{3}  F(0) (z^{3})}{3!} = - 3  \Big(\frac{z_1}{r}\Big)^{2} z \;\; \text{and} \;\; \frac{D^{4}  F(0) (z^{4})}{4!} = - 4 i \Big(\frac{z_1}{r}\Big)^{3} z.$$
  %  The above equations readily yields
    From (\ref{LmEqn}), the above equations yield
    $$ 2 \frac{\partial \rho}{\partial z}  \frac{D^{2}  F(0) (z^{2})}{2!}= 2 i \Big(\frac{z_1}{r}\Big) \rho (z), \;\;  2 \frac{\partial \rho}{\partial z} \frac{D^{3}  F(0) (z^{3})}{3!} = - 3 \Big(\frac{z_1}{r}\Big)^{2}\rho (z)  $$
    and
   $$  2 \frac{\partial \rho}{\partial z} \frac{D^{4}  F(0) (z^{4})}{4!} = -4i \Big(\frac{z_1}{r}\Big)^{3} \rho (z), $$
   respectively.  Taking $z = R u$ $(0< R <1)$, where $u =(u_1, u_2, \cdots, u_n)' \in \partial \Omega$ and $u_1 =r$, we obtain
\begin{equation}\label{FBs}
   2 \frac{\partial \rho}{\partial z}  \frac{D^{2}  F(0) (z^{2})}{2! \rho^2(z)}= 2 i , \;\;  2 \frac{\partial \rho}{\partial z} \frac{D^{3}  F(0) (z^{3})}{3! \ \rho^3(z)} =- 3 \;\; \text{and} \;\;  2 \frac{\partial \rho}{\partial z} \frac{D^{4}  F(0) (z^{4})}{4! \rho^4(z)} =- 4i.
\end{equation}
   Consequently, in view of (\ref{FBs}) and (\ref{Biis}), we deduce that
\begin{align*}
        \vert  (A_2 - A_4)( A_2^2-2 A_3^2+A_2 A_4)  \vert = 84 ,
\end{align*}
     showing that the bound is sharp.
\end{proof}
\begin{remark}
    When $n=1$ and $\Omega=\mathbb{U}$, Theorem~\ref{thm2} reduces to \hyperref[thmABC]{A}.
\end{remark}

\section*{Declarations}

\subsection*{Conflict of interest}
	The author declares that he has no conflict of interest.
\subsection*{Funding}
 Not applicable.
\subsection*{Author Contribution}
   The author solely contributed to the research and preparation of the manuscript.
\subsection*{Data Availability} Not applicable.

\end{document}